\documentclass[10pt]{amsart}
\usepackage{amsmath,amssymb,latexsym,soul,cite,amsthm,color,enumitem,graphicx,tikz, mathtools,microtype}
\usepackage[colorlinks=true,urlcolor=armygreen,citecolor=armygreen,linkcolor=armygreen,linktocpage,pdfpagelabels,bookmarksnumbered,bookmarksopen]{hyperref}
\definecolor{armygreen}{rgb}{0.29, 0.33, 0.13}
\usepackage[english]{babel}
\usepackage[left=3.2cm,right=3.2cm,top=2.9cm,bottom=2.9cm]{geometry}

\numberwithin{equation}{section}

\newtheorem{theorem}{Theorem}[section]
\theoremstyle{plain}
\newtheorem{lemma}[theorem]{Lemma}
\theoremstyle{plain}
\newtheorem{proposition}[theorem]{Proposition}
\theoremstyle{plain}

\theoremstyle{definition}
\newtheorem{remark}[theorem]{Remark}
\newtheorem{example}[theorem]{Example}

\newcommand{\R}{{\mathbb R}}
\newcommand{\eps}{\varepsilon}
\newcommand{\s}{H^s_0(\Omega)}

\newcommand{\beq}{\begin{equation}}
\newcommand{\eeq}{\end{equation}}
\renewcommand{\le}{\leqslant}
\renewcommand{\ge}{\geqslant}
\newcommand{\fl}{(-\Delta)^{s\,}}

\newcommand{\restr}[2]{\left.#1\right|_{#2}}

\newenvironment{enumroman}{\begin{enumerate}

}{\end{enumerate}}

\title[Resonant fractional problems]{Existence and multiplicity results for resonant fractional boundary value problems}

\author[A.\ Iannizzotto, N.S.\ Papageorgiou]{Antonio Iannizzotto, Nikolaos S.\ Papageorgiou}

\address[A.\ Iannizzotto]{Department of Mathematics and Computer Science
\newline\indent
University of Cagliari
\newline\indent
Viale L. Merello 92, 09123 Cagliari, Italy}
\email{antonio.iannizzotto@unica.it}

\address[N.S.\ Papageorgiou]{Department of Mathematics
\newline\indent
National Technical University
\newline\indent
Zografou Campus, Athens 15780, Greece}
\email{npapg@math.ntua.gr}


\begin{document}

\begin{abstract}
We study a Dirichlet-type boundary value problem for a pseudo-differential equation driven by the fractional Laplacian, with a non-linear reaction term which is resonant at infinity between two non-principal eigenvalues: for such equation we prove existence of a non-trivial solution. Under further assumptions on the behavior of the reaction at zero, we detect at least three non-trivial solutions (one positive, one negative, and one of undetermined sign). All results are based on the properties of weighted fractional eigenvalues, and on Morse theory.
\end{abstract}

\maketitle

\begin{center}
Version of \today\
\end{center}

\section{Introduction}\label{sec1}

\noindent
In the present paper we deal with the following Dirichlet-type boundary value problem:
\beq\label{dir}
\begin{cases}
\fl u=f(x,u) & \text{in $\Omega$} \\
u=0 & \text{in $\Omega^c$.}
\end{cases}
\eeq
Here $\Omega\subset\R^N$ ($N\ge 2$) is a bounded domain with a $C^2$-boundary $\partial\Omega$, $s\in(0,1)$, and the operator driving the equation is the Dirichlet fractional Laplacian, defined for any measurable function $u:\R^N\to\R$ and any $x\in\R^N$ by
\beq\label{fl}
\fl u(x)=C_{N,s}\lim_{\eps\to 0^+}\int_{\R^N\setminus B_\eps(x)}\frac{u(x)-u(y)}{|x-y|^{N+2s}}\,dy,
\eeq
where $C_{N,s}>0$ is a suitable normalization constant. Throughout the paper we will always assume $C_{N,s}=1$ (for a precise evaluation of $C_{N,s}$, consistent with alternative definitions of the fractional Laplacian, see \cite[Remark 3.11]{CS1}). Finally, $f:\Omega\times\R\to\R$ is a $C^1$-Carath\'eodory mapping, subject to various growth conditions both at zero and at infinity.
\vskip2pt
\noindent
The fractional Laplacian is the prototype of non-local pseudo-differential operator, and it appears in many interesting phenomena as the infinitesimal generator of L\'evy processes. For a general introduction to this operator, see \cite{CS1,CS2,C1,DPV}. When combined with a non-linear reaction term, the fractional Laplacian gives rise to non-local boundary value problems like \eqref{dir}, where the Dirichlet condition is stated on the complementary set $\Omega^c=\R^N\setminus\Omega$, instead of just on $\partial\Omega$, for both formal and intrinsic reasons, see \cite{BV,SV,SV1}. The problem admits a weak formulation and can be treated through variational methods and critical point theory: some results in this direction can be found in \cite{AP,BCSS,BMS,CW,CD,DI,FZM,F,IMS,IS,MR,MP,SV,T,WS} (see also \cite{ILPS,IMS1} for problems involving the non-linear corresponding operator, namely the fractional $p$-Laplacian). Just as in the classical case, knowing the asymptotic behavior (sub-linear, linear, or super-linear) of the reaction term $f(x,\cdot)$ gives precious information about the existence and multiplicity of solutions.
\vskip2pt
\noindent
While in \cite{DI} the cases of sub- or super-linear reactions at infinity were considered, here we want to focus on the case when $f(x,\cdot)$ is asymptotically linear at infinity, with a positive limit slope. The main tool in dealing with such reactions is a comparison with the corresponding eigenvalue problem (a similar approach was recently applied to a semilinear Robin problem, see \cite{PR}). We will see that the fractional Laplacian admits a divergent sequence of positive eigenvalues, thus decomposing the positive half-line $\R^+$ into a countable family of spectral intervals of the type $[\lambda_k,\lambda_{k+1}]$. We will assume that the limit slope of our reaction lies in one of such intervals, with $k\ge 1$ (thus, avoiding the coercive case). Note that we allow resonance at infinity, i.e., the case when the limit slope is exactly an eigenvalue. Then, we shall focus on the behavior of $f(x,\cdot)$ near zero:
\begin{itemize}[leftmargin=1cm]
\item[$(a)$] if $f(x,\cdot)$ is linear at zero as well, but with a slope lying in a different spectral interval, then by computing the critical groups at zero and at infinity we will prove existence of at least one non-trivial solution of problem \eqref{dir};
\item[$(b)$] if $f(x,\cdot)$ is super-linear at zero, or linear but with a smaller slope that the first eigenvalue, then by applying the Mountain Pass Theorem and a more sophisticated Morse-theoretic argument we will prove that \eqref{dir} admits at least three non-trivial solutions (one positive, one negative, and one with undetermined sign).
\end{itemize}
The structure of the paper is the following: in Section \ref{sec2} we recall some tools from critical point theory and Morse theory, and we establish as well some preliminary results on weighted eigenvalue problems for the fractional Laplacian; in Section \ref{sec3} we prove our existence result; and in Section \ref{sec4} we prove our multiplicity result.

\section{Preliminaries}\label{sec2}

\noindent
In this section we collect some definitions and results which will be useful in what follows.
\vskip2pt
\noindent
{\bf Notation:} We shall always use the $N$-dimensional Lebesgue measure on subsets of $\R^N$; for all $p\in[1,\infty]$ we denote by $\|\cdot\|_p$ the standard norm of $L^p(\Omega)$; $\delta_{i,j}$ denotes the Kronecker symbol; for all $t\in\R$ we set $t^\pm=\max\{0,\pm t\}$ (we shall use the same symbol for functions); in a Banach space, $\to$ will denote strong convergence and $\rightharpoonup$ weak convergence.

\subsection{Some basic facts from critical point and Morse theories}\label{sub21}

Let $(X,\|\cdot\|)$ be a real Hilbert space with topological dual $(X^*,\|\cdot\|_*)$, $\varphi\in C^1(X)$ be a functional. By $K(\varphi)$ we denote the set of all critical points of $\varphi$, i.e., those ponts $u\in X$ s.t. $\varphi'(u)=0$ in $X^*$, while for all $c\in\R$ we set
\[K_c(\varphi)=\{u\in K(\varphi):\,\varphi(u)=c\},\]
besides we set
\[\overline\varphi^c=\{u\in X:\,\varphi(u)\le c\}.\]
Most results in critical point theory require the following Cerami compactness condition (a weaker version of the Palais-Smale condition):
\begin{align}\tag{C}\label{c} 
\begin{split}
&\text{Any sequence $(u_n)$ in $X$, s.t.\ $(\varphi(u_n))$ is bounded in $\R$ and $(1+\|u_n\|)\varphi'(u_n)\to 0$}\\
&\text{in $X^*$, admits a (strongly) convergent subsequence.}
\end{split}
\end{align}
If $\varphi$ satisfies \eqref{c}, it is enough to prove that $\varphi$ admits a sequence $(u_n)$ as above to achieve a critical point.
\vskip2pt
\noindent
We recall some basic notions from Morse theory (see \cite{BSW,C,MMP}). Let $u\in K(\varphi)$ be an {\em isolated} critical point, i.e., there exists a neighborhood $U\subset X$ of $u$ s.t. $K(\varphi)\cap U=\{u\}$. Then, for all integer $q\ge 0$ the $q$-{\em th critical group of $\varphi$ at $u$} is defined as
\[C_q(\varphi,u)=H_q(\overline\varphi^c\cap U,\overline\varphi^c\cap U\setminus\{u\}),\]
where $H_q(\cdot,\cdot)$ is the $q$-th (singular) homology group of a topological pair (see \cite[Definition 6.9]{MMP}). Note that, by the excision property of homology groups, $C_q(\varphi,u)$ is invariant with respect to $U$.
\vskip2pt
\noindent
Different types of critical points can be distinguished by means of the critical groups:

\begin{proposition}\label{cg-min}
{\rm \cite[Example 6.45 (a)]{MMP}} Let $\varphi\in C^1(X)$ satisfy \eqref{c}, $u\in X$ be a local minimizer of $\varphi$ and an isolated critical point of $\varphi$. Then, for all $q\ge 0$
\[C_q(\varphi,u)=\delta_{q,0}\R.\]
\end{proposition}

\noindent
We recall that a critical point $u\in K(\varphi)$ is of {\em mountain pass type} if, for any neighborhood $U\subset X$ of $u$, the set
\[\big\{v\in U:\,\varphi(v)<\varphi(u)\big\}\]
is non-empty and path disconnected. We have the following result concerning the critical groups at such points:

\begin{proposition}\label{cg-mp1}
{\rm \cite[Proposition 6.100]{MMP}} Let $\varphi\in C^1(X)$ satisfy \eqref{c}, $u\in K(\varphi)$ be of mountain pass type. Then,
\[C_1(\varphi,u)\neq 0.\]
\end{proposition}

\noindent
Now assume that
\[\inf_{u\in K(\varphi)}\varphi(u)=:\bar c>-\infty.\]
Then we can as well define the $q$-{\em th critical group of $\varphi$ at infinity} as
\[C_k(\varphi,\infty)=H_k(X,\overline\varphi^c),\]
with $c<\bar c$ (this definition also is invariant with respect to $c$). All these groups are real linear spaces, whose dimensions are related by several results (the Morse relations). We recall two useful consequences of such relations:

\begin{proposition}\label{mr}
Let $\varphi\in C^2(X)$ satisfy \eqref{c}:
\begin{enumroman}
\item\label{mr1} {\rm \cite[Proposition 6.61 $(c)$]{MMP}} if $K(\varphi)=\{u_0\}$, then for all $q\ge 0$
\[C_q(\varphi,\infty)=C_q(\varphi,u_0);\]
\item\label{mr2} {\rm \cite[Proposition 6.89]{MMP}} if $K(\varphi)$ is a finite set containing $0$, and $k\ge 0$ an integer s.t.\ $C_k(\varphi,0)=0$ and $C_k(\varphi,\infty)\neq 0$, then there exists $u\in K(\varphi)$ s.t.\ $C_k(\varphi,u)\neq 0$.
\end{enumroman}
\end{proposition}

\noindent
Assume now that $\varphi\in C^2(X)$. For any $u\in K(\varphi)$, the {\em Morse index} of $\varphi$ at $u$ is defined as the supremum of dimensions of all linear subspaces of $X$ on which $\varphi''$ is negative definite, and the {\em nullity} of $\varphi$ at $u$ as the dimension of the kernel of $\varphi''(u)$ as a linear operator from $X$ into itself. In this framework, Proposition \ref{cg-mp1} has a stronger conclusion by means of the following result:

\begin{proposition}\label{cg-mp2}
{\rm \cite[Corollary 6.102]{MMP}} Let $\varphi\in C^2(X)$ satisfy \eqref{c}, $u\in K(\varphi)$ be s.t.\ both the Morse index $m_0$ and the nullity $\nu_0$ of $\varphi$ at $u$ are finite, $\nu_0\le 1$ whenever $m_0=0$, and $C_1(\varphi,u)\neq 0$. Then, for all $q\ge 0$
\[C_q(\varphi,u)=\delta_{q,1}\R.\]
\end{proposition}

\subsection{Fractional boundary value problems}\label{sub22}

Here we recall some basic results about \eqref{dir}-type problems. In all the forthcoming results we will assume the following hypotheses on the non-linearity:
\begin{itemize}[leftmargin=1cm]
\item[${\bf H}_0$] $f:\Omega\times\R\to\R$ is s.t.\ $f(\cdot,t)$ is measurable in $\Omega$ for all $t\in\R$, $f(x,\cdot)\in C^1(\R)$ for a.e.\ $x\in\Omega$, and there exist $a_0>0$, $p\in(1,2^*_s)$ s.t.\ for a.e. $x\in\Omega$ and all $t\in\R$
\[|f(x,t)|\le a_0(1+|t|^{p-1}).\]
\end{itemize}
First we define the function spaces that we are going to use. For all measurable function $u:\R^N\to\R$ we set
\[[u]_{s,2}^2=\iint_{\R^N\times\R^N}\frac{(u(x)-u(y))^2}{|x-y|^{N+2s}}\,dx\,dy,\]
then we define the fractional Sobolev space
\[H^s(\R^N)=\{u\in L^2(\R^N):\,[u]_{s,2}<\infty\}\]
(see \cite{DPV}). We restrict ourselves to the subspace
\[H^s_0(\Omega)=\{u\in H^s(\R^N):\,u(x)=0 \ \text{for a.e. $x\in\Omega^c$}\},\]
which is a separable Hilbert space under the norm $\|u\|=[u]_{s,2}$, induced by the inner product
\[\langle u,v\rangle =\iint_{\R^N\times\R^N}\frac{(u(x)-u(y))(v(x)-v(y))}{|x-y|^{N+2s}}\,dx\,dy\]
(see \cite{SV}). We denote by $(H^{-s}(\Omega),\|\cdot\|_*)$ the topological dual of $\s$. The critical exponent is defined as $2^*_s=2N(N-2s)^{-1}$, and the embedding $H^s_0(\Omega)\hookrightarrow L^p(\Omega)$ is continuous and compact for all $p\in[1,2^*_s)$ (see \cite[Lemma 8]{DPV}). Moreover, we introduce the positive order cone
\[H^s_0(\Omega)_+=\{u\in H^s_0(\Omega):\,u(x)\ge 0 \ \text{for a.e. $x\in\Omega$}\},\]
which has an empty interior with respect to the $H^s_0(\Omega)$-topology.
\vskip2pt
\noindent
We shall also need the weighted H\"older-type spaces defined as follows (see \cite{IMS}). Set $\delta(x)={\rm dist}(x,\Omega^c)$ for all $x\in\R^N$ and define
\[C^0_\delta(\overline\Omega)=\Big\{u\in C^0(\overline\Omega):\,\frac{u}{\delta^s}\in C^0(\overline\Omega)\Big\},\]
\[C^\alpha_\delta(\overline\Omega)=\Big\{u\in C^0(\overline\Omega):\,\frac{u}{\delta^s}\in C^\alpha(\overline\Omega)\Big\} \ (\alpha\in(0,1)),\]
endowed with the norms
\[\|u\|_{0,\delta}=\Big\|\frac{u}{\delta^s}\Big\|_\infty, \ \|u\|_{\alpha,\delta}=\|u\|_{0,\delta}+\sup_{x\neq y}\frac{|u(x)/\delta^s(x)-u(y)/\delta^s(y)|}{|x-y|^\alpha},\]
respectively. For all $0\le\alpha<\beta<1$ the embedding $C^\beta_\delta(\overline\Omega)\hookrightarrow C^\alpha_\delta(\overline\Omega)$ is continuous and compact. In this case, the positive cone $C^0_\delta(\overline\Omega)_+$ has a nonempty interior given by
\[{\rm int}\,(C^0_\delta(\overline\Omega)_+)=\Big\{u\in C^0_\delta(\overline\Omega):\,\frac{u(x)}{\delta^s(x)}>0 \ \text{for all $x\in\overline\Omega$}\Big\}.\]
The space $H^s_0(\Omega)$ provides the natural framework for the study of problem \eqref{dir}: a function $u\in H^s_0(\Omega)$ is a {\em (weak) solution} of \eqref{dir}, if for all $v\in H^s_0(\Omega)$
\[\langle u,v\rangle=\int_\Omega f(x,u)v\,dx.\]
Equivalently, we may define the energy functional $\varphi:H^s_0(\Omega)\to\R$ by setting for all $(x,t)\in\Omega\times\R$
\[F(x,t)=\int_0^t f(x,\tau)\,d\tau,\]
and for all $u\in H^s_0(\Omega)$
\beq\label{phi}
\varphi(u)=\frac{\|u\|^2}{2}-\int_\Omega F(x,u)\,dx.
\eeq
By ${\bf H}_0$ and the continuous embedding $H^s_0(\Omega)\hookrightarrow L^p(\Omega)$ we have $\varphi\in C^2(H^s_0(\Omega))$, and for all $u,v,w\in H^s_0(\Omega)$
\[\varphi'(u)(v)=\langle u,v\rangle-\int_\Omega f(x,u)v\,dx,\]
\[\varphi''(u)(v,w)=\langle v,w\rangle-\int_\Omega f'_t(x,u)vw\,dx.\]
So, $u$ is a solution of \eqref{dir} iff $\varphi'(u)=0$ in $H^{-s}(\Omega)$. We briefly discuss some regularity theory for the solutions of \eqref{dir},  starting with a simple {\em a priori} bound:

\begin{proposition}\label{apb}
{\rm \cite[Theorem 3.2]{IMS}} Let ${\bf H}_0$ hold, $u\in H^s_0(\Omega)$ be a solution of \eqref{dir}. Then $u\in L^\infty(\Omega)$ and
\[\|u\|_\infty\le M(\|u\|_{2^*_s}),\]
where $M\in C(\R^+,\R^+)$ is a non-decreasing function independent of $u$.
\end{proposition}

\noindent
While solutions of fractional equations exhibit good interior regularity properties, they may have a singular behavior on the boundary. We have the following global regularity result:

\begin{proposition}\label{reg}
Let ${\bf H}_0$ hold, $u\in H^s_0(\Omega)$ be a solution of \eqref{dir}. Then
\begin{enumroman}
\item\label{reg1} {\rm \cite[Corollary 5.6]{RS}} $u\in C^{1,\beta}(\Omega)$ for any $\beta\in(\max\{0,2s-1\},2s)$;
\item\label{reg2} {\rm \cite[Theorem 1.2]{RS}} $u\in C^\alpha_\delta(\overline\Omega)$ and
\[\|u\|_{\alpha,\delta}\le C(1+\|u\|_{2^*_s}),\]
with $\alpha\in(0,\min\{s,1-s\})$ and $C>0$ independent of $u$.
\end{enumroman}
\end{proposition}

\noindent
From \ref{reg2}, also recalling that $u=0$ in $\Omega^c$, one can see that the limit in \eqref{fl} exists in $\R$ and the equation in \eqref{dir} is satisfied in a pointwise sense (see \cite[Proposition 2.12]{IMS1}). We will also exploit the following fractional Hopf Lemma:

\begin{proposition}\label{hl}
{\rm \cite[Lemma 1.2]{GS1}} Let ${\bf H}_0$ hold, $f(x,t)\ge -ct$ for a.e. $x\in\Omega$ and all $t\in\R^+$ ($c>0$), and $u\in H^s_0(\Omega)_+$ be a solution of \eqref{dir}. Then, either $u=0$, or $u\in {\rm int}\,(C^0_\delta(\overline\Omega)_+)$.
\end{proposition}

\subsection{Fractional weighted eigenvalues}\label{sub23}

Here we focus on the following class of weighted fractional eigenvalue problems:
\beq\label{ev}
\begin{cases}
\fl u=\lambda\eta(x)u & \text{in $\Omega$} \\
u=0 & \text{in $\Omega^c$.}
\end{cases}
\eeq
Here $\eta\in L^\infty(\Omega)$ is a weight function s.t.\ $\eta(x)\ge\eta_0>0$ for a.e.\ $x\in\Omega$, while $\lambda>0$ is a real parameter. Clearly \eqref{ev} is a special case of problem \eqref{dir} with the linear reaction $f(x,t)=\lambda\eta(x)t$, which satisfies ${\bf H}_0$, so most of the results of Subsection \ref{sub22} apply. In particular, solutions of \eqref{ev} coincide with the critical points of the functional $\psi_{\eta,\lambda}\in C^2(\s)$ defined for all $u\in\s$ by
\[\psi_{\eta,\lambda}(u)=\frac{\|u\|^2}{2}-\frac{\lambda}{2}\int_\Omega\eta(x)u^2\,dx.\]
Let us fix $\eta$ as above. For all $\lambda>0$ we have $0\in K(\psi_{\eta,\lambda})$. If \eqref{ev} has a non-trivial solution $u\in\s\setminus\{0\}$, then we say that $\lambda$ is an {\em eigenvalue} and $u$ an associated {\em eigenfunction}. The set of all eigenvalues is the spectrum $\sigma(\eta)\subset\R^+_0$. The following result yields a complete description of the eigenpairs of \eqref{ev}, and follows closely the model of \cite[Proposition 9]{SV1} for the case $\eta\equiv 1$ (see also \cite{GS,PSY}). On the space $L^2(\Omega)$ we will use both the standard norm $\|\cdot\|_2$ and the equivalent weighted norm defined by
\[\|u\|_\eta^2=\int_\Omega\eta(x)u^2\,dx,\]
in which case we will write $L^2_\eta(\Omega)$.

\begin{proposition}\label{spe}
Let $\eta\in L^\infty(\Omega)$ be s.t.\ $\eta(x)\ge\eta_0>0$ for a.e.\ $x\in\Omega$. Then the set $\sigma(\eta)$ consists of a non-decreasing sequence
\[0<\lambda_1(\eta)<\lambda_2(\eta)\le\ldots\le\lambda_k(\eta)\le\ldots\]
with the following properties:
\begin{enumroman}
\item\label{spe1} the first eigenvalue $\lambda_1(\eta)$ is simple and isolated, with a unique positive eigenfunction $e_{1,\eta}\in\s$ s.t.\ $\|e_{1,\eta}\|_\eta=1$, with constant sign eigenfunctions, and
\[\lambda_1(\eta)=\inf_{u\in\s\setminus\{0\}}\frac{\|u\|^2}{\|u\|_\eta^2};\]
\item\label{spe2} for any $k\ge 2$, $\lambda_k(\eta)$ has a finite-dimensional eigenspace consisting of nodal eigenfunctions, among which $e_{k,\eta}\in\s$ s.t.\ $\|e_{k,\eta}\|_\eta=1$, and is given by the inductive formula
\[\lambda_k(\eta)=\inf_{u\in L_{k-1}^\bot(\eta)\setminus\{0\}}\frac{\|u\|^2}{\|u\|_\eta^2},\]
where $L_{k-1}^\bot(\eta)$ is the orthogonal complement of
\[L_{k-1}(\eta)={\rm span}\,\big\{e_{1,\eta},\ldots e_{k-1,\eta}\big\},\]
moreover the infimum above is attained exactly at $\lambda_k(\eta)$-eigenfunctions.
\item\label{spe3} $\lambda_k(\eta)\to\infty$ as $k\to\infty$;
\item\label{spe4} the sequence $(e_{k,\eta})$ is an orthogonal basis of $\s$ and an orthonormal basis of $L^2_\eta(\Omega)$.
\end{enumroman}
\end{proposition}
\begin{proof}
We prove \ref{spe1}. First we rephrase the definition of $\lambda_1(\eta)$ in the following equivalent form:
\[\lambda_1(\eta)=\inf_{u\in\mathcal{M}}\|u\|^2,\]
where
\[\mathcal{M}=\big\{u\in\s:\,\|u\|_\eta^2=1\big\}\]
is a $C^1$-manifold in $\s$. We prove that $\lambda_1(\eta)$ is attained on $\mathcal{M}$. Indeed, let $(u_n)$ be a minimizing sequence for $\lambda_1(\eta)$, in particular $(u_n)$ is bounded in $\s$, so passing to a subsequence we have $u_n\rightharpoonup\hat u_{1,\eta}$ in $\s$ and $u_n\to\hat u_{1,\eta}$ in $L^2(\Omega)$. Then clearly $\hat u_{1,\eta}\in\mathcal{M}$ and by convexity
\[\|\hat u_{1,\eta}\|^2\le\liminf_n\|u_n\|^2,\]
hence $\|\hat u_{1,\eta}\|^2=\lambda_1(\eta)$. By the Lagrange multiplier rule, we can find $\mu\in\R$ s.t.\ for all $v\in\s$
\[\langle\hat u_{1,\eta},v\rangle=\mu\int_\Omega\eta(x)\hat u_{1,\eta}v\,dx.\]
Choosing $v=\hat u_{1,\eta}$ in the identity above yields $\mu=\lambda_1(\eta)$, so we deduce that $\lambda_1(\eta)\in\sigma(\eta)$ with $\hat u_{1,\eta}$ as an associated eigenfunction. Since $\| |u| \|\le\|u\|$ for all $u\in\s$, we may assume $\hat u_{1,\eta}(x)\ge 0$ for a.e.\ $x\in\Omega$. Then, by Proposition \ref{hl}, we get $\hat u_{1,\eta}\in{\rm int}\,(C^0_\delta(\overline\Omega)_+)$. We normalize in $L^2_\eta(\Omega)$ by setting
\[e_{1,\eta}=\frac{\hat u_{1,\eta}}{\|\hat u_{1,\eta}\|_\eta}\in\s\cap{\rm int}\,(C^0_\delta(\overline\Omega)_+).\]
We prove now simplicity of the first eigenvalue, i.e., that $e_{1,\eta}$ is unique, arguing by contradiction: let $u\in\s\setminus\{e_{1,\eta}\}$ be another positive, $L^2_\eta(\Omega)$-normalized $\lambda_1(\eta)$-eigenfunction. Then $v=u-e_{1,\eta}$ is a $\lambda_1(\eta)$-eigenfunction as well, and as above it is shown to have constant sign in $\Omega$. For instance, let $v(x)\ge 0$ for all $x\in\Omega$, i.e., $u(x)\ge e_{1,\eta}(x)$ for all $x\in\Omega$, which implies
\[0<\int_\Omega\eta(x)(u^2-e^2_{1,\eta})\,dx=\|u\|_\eta^2-\|e_{1,\eta}\|_\eta^2=0,\]
a contradiction. So the eigenspace associated to $\lambda_1(\eta)$ is
\[L_1(\eta)={\rm span}\,\{e_{1,\eta}\},\]
consisting of constant sign functions. We shall see later that $\lambda_1(\eta)$ is isolated.
\vskip2pt
\noindent
Before going on, we need the following simple property: if $\lambda,\mu\in\sigma(\eta)$ with $\lambda\neq\mu$ and $u,v\in\s$ are a $\lambda$-eigenfunction and a $\mu$-eigenfunction, respectively, then
\beq\label{ort}
\langle u,v\rangle=\int_\Omega\eta(x) uv\,dx=0.
\eeq
Indeed, by the weak form of \eqref{ev} we have
\[\lambda\int_\Omega\eta(x)uv\,dx = \langle u,v\rangle = \mu\int_\Omega\eta(x)uv\,dx,\]
which by $\lambda\neq\mu$ implies
\[\int_\Omega\eta(x)uv\,dx = 0.\]
This in turn yields orthogonality in $\s$ by \eqref{ev} again.
\vskip2pt
\noindent
We prove \ref{spe2}, arguing inductively: let $k\ge 2$ and assume that $\lambda_1(\eta),\ldots \lambda_{k-1}(\eta)$ are already defined along with their eigenfunctions $e_{1,\eta},\ldots e_{k-1,\eta}$. Define $L_{k-1}(\eta)$ as above and set
\[\lambda_k(\eta)=\inf_{u\in\mathcal{M}_{k-1}}\|u\|^2,\]
where
\[\mathcal{M}_{k-1}=\big\{u\in L^\bot_{k-1}(\eta):\,\|u\|_\eta^2=1\big\}\]
is a $C^1$-manifold in the Hilbert space $L^\bot_{k-1}(\eta)$. Arguing as above we see that $\lambda_k(\eta)\in\sigma(\eta)$ with an associated, $L^2_\eta(\Omega)$-normalized eigenfunction $e_{k,\eta}\in L^\bot_{k-1}(\eta)$. Clearly, since $L^\bot_{k-1}(\eta)\subseteq L^\bot_{k-2}(\eta)$ we have $\lambda_k(\eta)\ge\lambda_{k-1}(\eta)$.
\vskip2pt
\noindent
Strict monotonicity can only be proved for $k=2$. Indeed, if we assume $\lambda_1(\eta)=\lambda_2(\eta)$, then there exists $\tau\in\R$ s.t.\ $e_{2,\eta}=\tau e_{1,\eta}\neq 0$, against $e_{2,\eta}\in L^\bot_1(\eta)$.
\vskip2pt
\noindent
Thus, for all $k\ge 2$ we have
\[\lambda_k(\eta)\ge\lambda_2(\eta)>\lambda_1(\eta)\]
and by \eqref{ort} we have
\[\int_\Omega\eta(x) e_{k,\eta}e_{1,\eta}\,dx=0,\]
which implies that $e_{k,\eta}$ is nodal. We prove that the $\lambda_k(\eta)$-eigenspace is finite-dimensional, arguing by contradiction: assume that $(u_n)$ is a sequence of linearly independent $\lambda_k(\eta)$-eigenfunctions. Since $\s$ is a Hilbert space, we may assume without loss of generality that for all integer $n\neq m$
\[\|u_n\|^2=\lambda_k(\eta), \ \|u_n\|_\eta^2=1, \ \int_\Omega\eta(x) u_n u_m\,dx=0.\]
In particular, $(u_n)$ is bounded in $\s$. Passing to a subsequence, $(u_n)$ is convergent in $L^2_\eta(\Omega)$, i.e., it is a Cauchy sequence in $L^2_\eta(\Omega)$. But for all $n\neq m$ we have
\[\|u_n-u_m\|_\eta^2=\|u_n\|_\eta^2+\|u_m\|_\eta^2=2,\]
a contradiction. Thus the $\lambda_k(\eta)$-eigenspace has infinite dimension.
\vskip2pt
\noindent
We prove \ref{spe3}, arguing by contradiction: assume that the sequence $(\lambda_k(\eta))$ is bounded in $\R$, then $(e_{k,\eta})$ is bounded in $\s$ and as in the proof of \ref{spe2} we reach a contradiction.
\vskip2pt
\noindent
Finally we prove \ref{spe4}. Orthogonality in $\s$ and orthonormality in $L^2_\eta(\Omega)$ follow from \eqref{ort}. To see that $(e_{k,\eta})$ is a basis for $\s$ it suffices to show that if $v\in\s$ is s.t.\ for all $k\ge 1$
\[\langle e_{k,\eta},v\rangle=0,\]
then $v=0$. We argue by contradiction, assuming on the contrary that there exists $v\in{\rm span}\,(e_{k,\eta})^\bot$ s.t. $\|v\|_\eta=1$. By \ref{spe3} there exists $h\ge 1$ s.t.\ $\lambda_h(\eta)>\|v\|^2$, which by \ref{spe2} implies $v\notin L^\bot_{h-1}(\eta)$, a contradiction. An easy density argument shows that $(e_{k,\eta})$ is basis for $L^2_\eta(\Omega)$ as well.
\vskip2pt
\noindent
To conclude the proof, it remains to prove that there is no other eigenvalue than the elements of the sequence $(\lambda_k(\eta))$. We argue by contradiction, assuming that there exists $\lambda\in\sigma(\eta)$ s.t.\ $\lambda\neq\lambda_k(\eta)$ for all $k\ge 1$. Let $u\in\s$ be a $\lambda$-eigenfunction, then we have
\[\|u\|^2=\lambda\|u\|_\eta^2.\]
By \ref{spe1} we have $\lambda>\lambda_1$, hence there exists $k\ge 1$ s.t.\ $\lambda_k(\eta)<\lambda<\lambda_{k+1}(\eta)$. We claim that $u\notin L^\bot_k(\eta)$. Indeed, otherwise by \ref{spe2} we have
\[\frac{\|u\|^2}{\|u\|_\eta^2}\ge\lambda_{k+1}(\eta)>\lambda,\]
a contradiction. Thus, there exists $j\in\{1,\ldots k\}$ s.t.
\[\langle e_{j,\eta},u\rangle\neq 0,\]
against \eqref{ort}. In particular, we have $\sigma(\eta)\cap(\lambda_1(\eta),\lambda_2(\eta))=\emptyset$, i.e., $\lambda_1(\eta)$ is an isolated element of $\sigma(\eta)$. That concludes the proof.
\end{proof}

\begin{remark}
We count eigenvalues with their multiplicity, so for $k\ge 2$ we may have
\[\lambda_k(\eta)=\lambda_{k+1}(\eta)=\ldots=\lambda_{k+h}(\eta),\]
each integer from $k$ to $k+h$ corresponding to a new, linearly independent eigenfunction. Also, note that all issues in Proposition \ref{spe} above can be equivalently stated in $L^2(\Omega)$ (with the ordinary norm), but the orthogonality in \ref{spe3}.
\end{remark}

\noindent
When $\eta\equiv 1$ we will just write $\sigma(1)=\sigma$, $\lambda_k(1)=\lambda_k$, and $e_{k,1}=e_k$. If the weight function is smooth enough we have the so-called {\em unique continuation property (u.c.p.)} for eigenfunctions (a consequence of \cite[Theorem 1.3]{FF}):

\begin{proposition}\label{ucp}
Let $\eta\in L^\infty(\Omega)\cap C^1(\Omega)$ be s.t.\ $\eta(x)\ge\eta_0>0$ for all $x\in\Omega$, moreover let $\lambda\in\sigma(\eta)$ and $u\in\s\setminus\{0\}$ be a $\lambda$-eigenfunction. Then the set
\[\big\{x\in\Omega:\,u(x)=0\big\}\]
has zero measure.
\end{proposition}

\noindent
As in the classical case (see \cite{DG}), Proposition \ref{ucp} can be used to prove monotonicity of the mapping $\eta\mapsto\lambda_k(\eta)$, for any $k\ge 1$:

\begin{proposition}\label{mon}
Let $\eta_1,\eta_2\in L^\infty(\Omega)$ be s.t.\ $\eta_2(x)\ge\eta_1(x)\ge\eta_0>0$ for a.e.\ $x\in\Omega$, $\eta_1\not\equiv\eta_2$, and either $\eta_1\in C^1(\Omega)$ or $\eta_2\in C^1(\Omega)$. Then, $\lambda_k(\eta_1)>\lambda_k(\eta_2)$ for all $k\ge 1$.
\end{proposition}
\begin{proof}
Assume that $\eta_1\in C^1(\Omega)$ (the other case is treated similarly), so by Proposition \ref{ucp} all $\lambda_k(\eta_1)$-eigenfunctions have u.c.p. We recall the following variational characterization of $\eta$-eigenvalues, which is equivalent to that of Proposition \ref{spe}:
\beq\label{alt}
\frac{1}{\lambda_k(\eta_i)}=\sup_{F\in\mathcal{F}_k}\,\inf_{u\in F,\,\|u\|=1}\int_\Omega\eta_i(x)u^2\,dx \ (i=1,2),
\eeq
where $\mathcal{F}_k$ denotes the family of all $k$-dimensional linear subspaces of $\s$. The supremum in \eqref{alt} (with $i=1$) is attained at a certain subspace $F\in\mathcal{F}_k$, namely
\[\frac{1}{\lambda_k(\eta_1)}=\inf_{u\in F,\,\|u\|=1}\int_\Omega\eta_1(x)u^2\,dx.\]
For any $u\in F$, $\|u\|=1$, two cases may occur:
\begin{itemize}[leftmargin=1cm]
\item[$(a)$] either
\[\frac{1}{\lambda_k(\eta_1)}=\int_\Omega\eta_1(x)u^2\,dx,\]
then $u$ is a $\lambda_k(\eta_1)$-eigenfunction, enjoying u.c.p., hence by the assumptions on $\eta_1$ and $\eta_2$ we have
\[\frac{1}{\lambda_k(\eta_1)}=\int_\Omega\eta_1(x)u^2\,dx<\int_\Omega\eta_2(x)u^2\,dx;\]
\item[$(b)$] or simply
\[\frac{1}{\lambda_k(\eta_1)}<\int_\Omega\eta_1(x)u^2\,dx\le\int_\Omega\eta_2(x)u^2\,dx.\]
\end{itemize}
In both cases we have
\[\frac{1}{\lambda_k(\eta_1)}<\int_\Omega\eta_2(x)u^2\,dx\]
for all $u\in F$, $\|u\|=1$. By compactness (recall that $F$ is finite-dimensional) and using \eqref{alt} with $i=2$, we have
\[\frac{1}{\lambda_k(\eta_1)}<\inf_{u\in F,\,\|u\|=1}\int_\Omega\eta_2(x)u^2\,dx\le\frac{1}{\lambda_k(\eta_2)},\]
i.e., $\lambda_k(\eta_1)>\lambda_k(\eta_2)$.
\end{proof}

\noindent
We conclude this section with a result on critical groups of the functional $\psi_{\eta,\lambda}$:

\begin{proposition}\label{evcg}
Let $\eta\in L^\infty(\Omega)$ be s.t.\ $\eta(x)\ge\eta_0>0$ for a.e.\ $x\in\Omega$, $k\ge 1$ s.t.\ $\lambda_k(\eta)<\lambda_{k+1}(\eta)$, and $\lambda\in(\lambda_k(\eta),\lambda_{k+1}(\eta))$. Then for all $\ge 0$
\begin{enumroman}
\item\label{evcg1} $C_q(\psi_{\eta,\lambda},0)=\delta_{q,k}\R$;
\item\label{evcg2} $C_q(\psi_{\eta,\lambda},\infty)=\delta_{q,k}\R$.
\end{enumroman}
\end{proposition}
\begin{proof}
By Proposition \ref{ev} we know that $\lambda\notin\sigma(\eta)$. So $K(\psi_{\eta,\lambda})=\{0\}$, and the Morse index of $\psi_{\eta,\lambda}$ at $0$ is $k$. Thus, by \cite[Theorem 6.51]{MMP} (see also \cite[Proposition 2.3]{S}) we have for all $q\ge 0$
\[C_q(\psi_{\eta,\lambda},0)=\delta_{q,k}\R,\]
which proves \ref{evcg1}. By Proposition \ref{mr} \ref{mr1}, we also have for all $q\ge 0$
\[C_q(\psi_{\eta,\lambda},\infty)=\delta_{q,k}\R,\]
which proves \ref{evcg2}.
\end{proof}

\section{Existence result}\label{sec3}

\noindent
This section is devoted to proving our existence result for problem \eqref{dir}. In this connection, our hypotheses on the reaction $f$ are the following:

\begin{itemize}[leftmargin=1cm]
\item[${\bf H}_1$] $f:\Omega\times\R\to\R$ is s.t.\ $f(\cdot,t)$ is measurable in $\Omega$ for all $t\in\R$, $f(x,\cdot)\in C^1(\R)$ for a.e.\ $x\in\Omega$, we set for all $(x,t)\in\Omega\times\R$
\[F(x,t)=\int_0^t f(x,\tau)\,d\tau.\]
Morerover:
\begin{enumroman}
\item\label{h11} for all $\rho>0$ there exists $a_\rho\in L^\infty(\Omega)_+$ s.t.\ $|f(x,t)|\le a_\rho(x)$ for a.e.\ $x\in\Omega$ and all $|t|\le\rho$;
\item\label{h12} $\displaystyle\lim_{|t|\to\infty}\big(f(x,t)t-2F(x,t)\big)=+\infty$ uniformly for a.e.\ $x\in\Omega$;
\item\label{h13} there exists an integer $k\ge 1$ s.t.\ uniformly for a.e.\ $x\in\Omega$
\[\lambda_k\le\liminf_{|t|\to \infty}\frac{f(x,t)}{t}\le\limsup_{|t|\to\infty}\frac{f(x,t)}{t}\le\lambda_{k+1};\]
\item\label{h14} there exist an integer $h\ge 1$, s.t.\ $\lambda_k\neq\lambda_h$, functions $\eta_1,\eta_2\in L^\infty(\Omega)$ s.t.\ $\lambda_h\le\eta_1(x)\le\eta_2(x)\le\lambda_{h+1}$ for a.e.\ $x\in\Omega$, with $\eta_1\not\equiv\lambda_h$, $\eta_2\not\equiv\lambda_{h+1}$, and $\delta_0>0$ s.t.\ for a.e.\ $x\in\Omega$ and all $|t|\le\delta_0$
\[\eta_1(x)\le\frac{f(x,t)}{t}\le\eta_2(x).\]
\end{enumroman}
\end{itemize}

\begin{example}
Let $f\in C^1(\R)$ be defined by
\[f(t)=\begin{cases}
\lambda_k t-(\mu-\lambda_k)\Big(\frac{1}{2}\ln|t|+\sqrt{|t|}\big) & \text{if $t<-1$} \\
\mu t & \text{if $|t|\le 1$} \\
\lambda_k t+(\mu-\lambda_k)\Big(\frac{1}{2}\ln|t|+\sqrt{|t|}\big) & \text{if $t>1$,}
\end{cases}\]
where $\mu\in(\lambda_h,\lambda_{h+1})$, and $h<k$ are integers. Then $f$ satisfies ${\bf H}_1$ \ref{h11} - \ref{h14}.
\end{example}

\noindent
Clearly, from ${\bf H}_1$ \ref{h11}, \ref{h13} it follows that $f$ satisfies ${\bf H}_0$. Besides, from ${\bf H}_1$ \ref{h14} we see at once that $f(x,0)=0$ for a.e.\ $x\in\Omega$. We define $\varphi\in C^2(\s)$ as in Subsection \ref{sub22} and we establish some structural properties:

\begin{lemma}\label{ex-c}
If ${\bf H}_1$ \ref{h11} - \ref{h13} hold, then $\phi$ satisfies \eqref{c}.
\end{lemma}
\begin{proof}
Let $(u_n)$ be a sequence in $\s$ s.t.\ $|\varphi(u_n)|\le M$ for all $n\ge 1$ and $(1+\|u_n\|)\varphi'(u_n)\to 0$ in $H^{-s}(\Omega)$. We can find a real sequence $(\eps_n)$ s.t.\ $\eps_n\to 0^+$ and for all $n\ge 1$, $v\in\s$
\beq\label{ex-c1}
\Big|\langle u_n,v\rangle-\int_\Omega f(x,u_n)v\,dx\Big|\le\frac{\eps_n\|v\|}{1+\|u_n\|}.
\eeq
Choosing $v=u_n$ in \eqref{ex-c1} yields
\[-\|u_n\|^2+\int_\Omega f(x,u_n)u_n\,dx\le\eps_n,\]
while from boundedness of $(\varphi(u_n))$ we get
\[\|u_n\|^2-2\int_\Omega F(x,u_n)\,dx\le 2M.\]
Adding up we get for all integer $n\ge 1$
\beq\label{ex-c2}
\int_\Omega\big(f(x,u_n)u_n-2F(x,u_n)\big)\,dx\le M' \ (M'>0).
\eeq
We claim that $(u_n)$ is bounded in $\s$. Arguing by contradiction, and passing if necessary to a subsequence, we may assume that $\|u_n\|\to\infty$. Set for all $n\ge 1$ $w_n=u_n\|u_n\|^{-1}$, so we have $\|w_n\|=1$. Passing to a subsequence we may assume that $w_n\rightharpoonup w$ in $\s$ and $w_n\to w$ in $L^2(\Omega)$ (by the compact embedding $\s\hookrightarrow L^2(\Omega)$). From \eqref{ex-c2} we have for all $n\ge 1$, $v\in\s$
\beq\label{ex-c3}
\Big|\langle w_n,v\rangle-\int_\Omega \frac{f(x,u_n)}{\|u_n\|}v\,dx\Big|\le\frac{\eps_n\|v\|}{\|u_n\|+\|u_n\|^2}.
\eeq
We note that the sequence $(f(\cdot,u_n)\|u_n\|^{-1})$ is bounded in $L^2(\Omega)$. Indeed, by ${\bf H}_1$ \ref{h13} we can find $\beta,\rho>0$ s.t.\ for a.e.\ $x\in\Omega$ and all $|t|>\rho$
\[|f(x,t)|\le\beta|t|.\]
For all $n\ge 1$ set
\[\Omega'_n=\big\{x\in\Omega:\,|u_n(x)|\le\rho\big\}, \ \Omega''_n=\Omega\setminus\Omega'_n,\]
so by ${\bf H}_1$ \ref{h11} we have
\begin{align*}
\int_\Omega\frac{f(x,u_n)^2}{\|u_n\|^2}\,dx &\le \frac{1}{\|u_n\|^2}\int_{\Omega'_n}a_\rho(x)^2\,dx + \frac{\beta^2}{\|u_n\|^2}\int_{\Omega''_n}u_n^2\,dx \\
&\le \frac{\|a_\rho\|_2^2}{\|u_n\|^2}+\beta^2\frac{\|u_n\|_2^2}{\|u_n\|^2},
\end{align*}
and the latter is bounded (due to the continuous embedding $\s\hookrightarrow L^2(\Omega)$). Passing to a subsequence, we may assume that $(f(\cdot,u_n)\|u_n\|^{-1})$ is weakly convergent in $L^2(\Omega)$. More precisely, by ${\bf H}_1$ \ref{h13} we can find $\eta_\infty\in L^\infty(\Omega)$ s.t.\ $\lambda_k\le\eta_\infty(x)\le\lambda_{k+1}$ for a.e.\ $x\in\Omega$ and
\beq\label{ex-c4}
\frac{f(\cdot,u_n)}{\|u_n\|}\rightharpoonup \eta_\infty w \ \text{in $L^2(\Omega)$.}
\eeq
Choosing $v=w_n-w$ in \eqref{ex-c3} and using H\"older inequality yields for all $n\ge 1$
\[\big|\langle w_n,w_n-w\rangle\big| \le \Big\|\frac{f(\cdot,u_n)}{\|u_n\|}\Big\|_2 \, \|w_n-w\|_2+\frac{\eps_n\|w_n-w\|}{\|u_n\|+\|u_n\|^2},\]
and the latter tends to $0$ as $n\to\infty$. Therefore we have
\[\|w_n-w\|^2=\langle w_n,w_n-w\rangle-\langle w,w_n-w\rangle \to 0,\]
i.e., $w_n\to w$ in $\s$. In particular we have $\|w\|=1$. Passing to the limit in \eqref{ex-c3} and using \eqref{ex-c4}, we see that $w$ is a solution of the \eqref{ev}-type problem
\beq\label{ex-c5}
\begin{cases}
\fl w=\eta_\infty(x)w & \text{in $\Omega$} \\
w=0 & \text{in $\Omega^c$.}
\end{cases}
\eeq
So, $1\in\sigma(\eta_\infty)$ with $w$ as an associated eigenfunction. Three cases may occur:
\begin{itemize}[leftmargin=1cm]
\item[$(a)$] if $\eta_\infty\not\equiv\lambda_k,\,\lambda_{k+1}$ (non-resonance), then by Proposition \ref{mon} (recall that any constant weight is $C^1(\Omega)$) we have
\[\lambda_k(\eta_\infty)<\lambda_k(\lambda_k)=1=\lambda_{k+1}(\lambda_{k+1})<\lambda_{k+1}(\eta_\infty),\]
against the fact that $1$ is an eigenvalue of \eqref{ex-c5};
\item[$(b)$] if $\eta_\infty\equiv\lambda_k$ (left resonance), then by Proposition \ref{ucp} $w$ has u.c.p., that is, $w(x)\neq 0$ for a.e.\ $x\in\Omega$, which implies $|u_n(x)|\to\infty$ for a.e.\ $x\in\Omega$, hence by ${\bf H}_1$ \ref{h12} and the Fatou Lemma we have
\[\lim_n\int_\Omega\big(f(x,u_n)u_n-2F(x,u_n)\big)\,dx=\infty,\]
against \eqref{ex-c2};
\item[$(c)$] if $\eta_\infty\equiv\lambda_{k+1}$ (right resonance), then reasoning as in case $(b)$ we reach a contradiction.
\end{itemize}
Thus, $(u_n)$ is bounded in $\s$. Passing to a subsequence, we have $u_n\rightharpoonup u$ in $\s$ and $u_n\to u$ in $L^2(\Omega)$. Choosing $v=u_n-u$ in \eqref{ex-c1} and letting $n\to\infty$, we get
\[\lim_n\langle u_n,u_n-u\rangle=0,\]
and as above we deduce $u_n\to u$ in $\s$, concluding the proof.
\end{proof}

\noindent
Now we aim at a description of the critical set of $\varphi$: since we are looking for a non-trivial solution, we can always assume that $K(\varphi)$ is a {\em finite} set. We compute the critical groups of $\varphi$ at infinity:

\begin{lemma}\label{ex-inf}
If ${\bf H}_1$ \ref{h11} - \ref{h13} hold, then for all $q\ge 0$
\[C_q(\varphi,\infty)=\delta_{q,k}\R.\]
\end{lemma}
\begin{proof}
Fix $\lambda\in(\lambda_k,\lambda_{k+1})$ and define $\psi_\lambda=\psi_{1,\lambda}$ as in Subsection \ref{sub23}. By Proposition \ref{evcg} \ref{evcg2} we have for all integer $q\ge 0$
\beq\label{ex-inf1}
C_q(\psi_\lambda,\infty)=\delta_{q,k}\R.
\eeq
Now define a homotopy by setting for all $t\in[0,1]$, $u\in\s$
\[h(t,u)=(1-t)\varphi(u)+t\psi_\lambda(u).\]
Clearly $h(t,\cdot)\in C^1(\s)$ for all $t\in[0,1]$, $h(0,\cdot)=\varphi$, and $h(1,\cdot)=\psi_\lambda$. Moreover, $h$ satisfies a uniform \eqref{c}-type condition at infinity: there exist $\beta,\rho>0$ s.t.\ for all $(t,u)\in[0,1]\times\s$
\beq\label{ex-inf2}
h(t,u)\le\beta \ \Longrightarrow \ (1+\|u\|)\|h'_u(t,u)\|_*\ge\rho.
\eeq
Arguing by contradiction, we assume that there exist sequences $(t_n)$ in $[0,1]$, $(u_n)$ in $\s$ s.t.\ $h(t_n,u_n)\to -\infty$ and $(1+\|u\|)h'_u(t_n,u_n)\to 0$ in $H^{-s}(\Omega)$. Set for all $n\ge 1$ $w_n=u_n\|u_n\|^{-1}$, so $\|w_n\|=1$ and there exists $(\eps_n)$ s.t.\ $\eps_n\to 0^+$ and for all $n\ge 1$, $v\in\s$ we have
\beq\label{ex-inf3}
\Big|\langle w_n,v\rangle-\int_\Omega\Big((1-t_n)\frac{f(x,u_n)}{\|u_n\|}+t_n \lambda w_n\Big)v\,dx\Big|\le\frac{\eps_n\|v\|}{\|u_n\|+\|u_n\|^2}.
\eeq
Passing to a subsequence, we have $t_n\to t$ in $\R$, $w_n\rightharpoonup w$ in $\s$, and $w_n\to w$ in $L^2(\Omega)$, as well as $f(\cdot,u_n)\|u_n\|^{-1}\rightharpoonup\eta_\infty w$ in $L^2(\Omega)$ for some $\eta_\infty\in L^\infty(\Omega)$ with $\lambda_k\le\eta_\infty(x)\le\lambda_{k+1}$ a.e.\ in $\Omega$. Reasoning as in the proof of Lemma \ref{ex-c} we deduce that $w_n\to w$ in $\s$, in particular $\|w\|=1$. Passing to the limit in \eqref{ex-inf3} we see that $w$ is a solution of
\beq\label{ex-inf4}
\begin{cases}
\fl w=\big((1-t)\eta_\infty(x)+t\lambda\big)w & \text{in $\Omega$} \\
w=0 & \text{in $\Omega^c$.}
\end{cases}
\eeq
So, $1\in\sigma\big((1-t)\eta_\infty+t\lambda\big)$ with $w$ as an associated eigenfunction. We note that $t<1$, otherwise the weight function in \eqref{ex-inf4} would reduce to $\lambda$, against our choice and the fact that $\lambda\notin\sigma$. Now we distinguish three cases as in the proof of Lemma \ref{ex-c} and we reach a contradiction in any of such cases. Thus \eqref{ex-inf2} is achieved. We invoke \cite[Proposition 3.2]{LS} to have homotopy invariance of the critical groups at infinity, i.e., for all $q\ge 0$
\[C_q(\varphi,\infty)=C_q(\psi_\lambda,\infty).\]
By \eqref{ex-inf1} we conclude.
\end{proof}

\noindent
It is a far more delicate matter to compute the critical groups at $0$:

\begin{lemma}\label{ex-z}
If ${\bf H}_1$ \ref{h11} - \ref{h14} hold, then for all $q\ge 0$
\[C_q(\varphi,0)=\delta_{q,h}\R.\]
\end{lemma}
\begin{proof}
By ${\bf H}_1$ \ref{h14}, clearly $f(x,0)=0$ for a.e.\ $x\in\Omega$, so $0\in K(\varphi)$. Since by assumption $K(\varphi)$ is finite, this is an isolated critical point. Fix $\lambda\in(\lambda_h,\lambda_{h+1})$ and define $\psi_\lambda\in C^2(\s)$ as in the proof of Lemma \ref{ex-inf}. By Proposition \ref{evcg} \ref{evcg1} we have $0\in K(\psi_\lambda)$ and for all integer $q\ge 0$
\beq\label{ex-z1}
C_q(\psi_\lambda,0)=\delta_{q,h}\R.
\eeq
Preliminarily we note that Proposition \ref{spe} ensures the orthogonal decomposition $\s=\check H\oplus\hat H$, where
\[\check H={\rm span}\,\{e_1,\ldots e_h\}, \ \hat H=\overline{\rm span}\,\{e_j:\,j\ge h+1\}\]
are both Hilbert spaces ($\check H$ with finite dimension). Thus, for all $u\in\s$ there exists unique $\check u\in\check H$, $\hat u\in\hat H$ s.t.\ $u=\check u+\hat u$, moreover the following equality holds for a.e.\ $x\in\Omega$:
\beq\label{ex-z2}
u(x)(\hat u(x)-\check u(x))=\hat u(x)^2-\check u(x)^2.
\eeq
We restrict ourselves to $X=\s\cap C^0_\delta(\overline\Omega)$, which is a Banach space, dense in $\s$, endowed with the norm
\[\|u\|_X=\|u\|+\|u\|_{0,\delta}\]
(inducing a stronger topology than that of $C^0_\delta(\overline\Omega)$). We denote
\[\tilde\varphi=\restr{\varphi}{X}, \ \tilde\psi_\lambda=\restr{\psi_\lambda}{X}.\]
Then clearly $\tilde\varphi,\tilde\psi_\lambda\in C^1(X)$ and their critical groups at $0$ coincide with those of $\varphi$, $\psi_\lambda$, respectively (see \cite[p.\ 14]{C}). For all $(t,u)\in[0,1]\times X$ we set
\[\tilde h(t,u)=(1-t)\tilde\varphi(u)+t\tilde\psi_\lambda(u),\]
so for all $t\in [0,1]$ we have $\tilde h(t,\cdot)\in C^1(X)$ and $0\in K(\tilde h(t, \cdot))$. We claim now that $0$ is an isolated critical point of $\tilde h(t,\cdot)$, uniformly with respect to $t\in[0,1]$. Let $\delta_0>0$ be as in ${\bf H}_1$ \ref{h14}. For any $u\in X$, $\|u\|_X\le\delta_0$ we have in particular $\|u\|_\infty\le\delta_0$. We claim that for a.e.\ $x\in\Omega$
\beq\label{ex-z3}
f(x,u(x))(\hat u(x)-\check u(x))\le\eta_2(x)\hat u(x)^2-\eta_1\check u(x)^2.
\eeq
Indeed, for a.e.\ $x\in\Omega$ two cases may occur:
\begin{itemize}[leftmargin=1cm]
\item[$(a)$] if $u(x)(\hat u(x)-\check u(x))\ge 0$, then by \eqref{ex-z2} and ${\bf H}_1$ \ref{h14} we have
\begin{align*}
& f(x,u(x))(\hat u(x)-\check u(x))\le\eta_2(x)u(x)(\hat u(x)-\check u(x)) \\
&=\eta_2(x)(\hat u(x)^2-\check u(x)^2)\le\eta_2(x)\hat u(x)^2-\eta_1\check u(x)^2;
\end{align*}
\item[$(b)$] if $u(x)(\hat u(x)-\check u(x))<0$, then by \eqref{ex-z2} and ${\bf H}_1$ \ref{h14} we have
\begin{align*}
& f(x,u(x))(\hat u(x)-\check u(x))\le\eta_1(x)u(x)(\hat u(x)-\check u(x)) \\
&=\eta_1(x)(\hat u(x)^2-\check u(x)^2)\le\eta_2(x)\hat u(x)^2-\eta_1\check u(x)^2.
\end{align*}
\end{itemize}
By \eqref{ex-z3} we have
\begin{align}\label{ex-z4}
\langle\tilde\varphi(u),\hat u-\check u\rangle &= \langle u,\hat u-\check u\rangle-\int_\Omega f(x,u)(\hat u-\check u)\,dx \\
\nonumber&\ge \|\hat u\|^2-\|\check u\|^2-\int_\Omega\big(\eta_2(x)\hat u^2-\eta_1(x)\check u^2\big)\,dx \\
\nonumber&= \tilde\gamma_2(\hat u)-\tilde\gamma_1(\check u),
\end{align}
where we have set for all $v\in X$
\[\tilde\gamma_i(v)=\|v\|^2-\int_\Omega\eta_i(x)v^2\,dx \ (i=1,2).\]
To proceed, we need some estimates on $\tilde\gamma_i$. First, we prove that there exists $c_1>0$ s.t.\ for all $v\in X\cap\check H$
\beq\label{ex-z5}
\tilde\gamma_1(v)\le -c_1\|v\|^2.
\eeq
Indeed, since $\eta_1(x)\ge\lambda_h$ for a.e.\ $x\in\Omega$, and by Proposition \ref{spe} \ref{spe2} we have
\[\tilde\gamma_1(v)\le\|v\|^2-\lambda_h\|v\|_2^2\le 0.\]
To get a strict inequality, we argue by contradiction: assume that $(v_n)$ is a sequence in $X\cap\check H$ s.t.\ $\|v_n\|=1$ for all integer $n\ge 1$ and $\tilde\gamma_1(v_n)\to 0$. Since $\check H$ is finite-dimensional, passing to a subsequence we have $v_n\to v$ in $\s$, hence $\|v\|=1$ and $\tilde\gamma_1(v)=0$, so
\[\|v\|^2=\int_\Omega\eta_1(x)v^2\,dx\ge\lambda_h\|v\|_2^2,\]
against $v\in\check H$.
\vskip2pt
\noindent
Besides, there exists $c_2>0$ s.t.\ for all $v\in X\cap\hat H$
\beq\label{ex-z6}
\tilde\gamma_2(v)\ge c_2\|v\|^2.
\eeq
Indeed, since $\eta_2(x)\le\lambda_{h+1}$ for a.e.\ $x\in\Omega$, and by Proposition \ref{spe} \ref{spe2} we have
\[\tilde\gamma_2(v)\ge\|v\|^2-\lambda_{h+1}\|v\|^2_2\ge 0.\]
To complete the proof of \eqref{ex-z6}, again we argue by contradiction: let $(v_n)$ be a sequence in $X\cap\hat H$ s.t.\ $\|v_n\|=1$ for all $n\ge 1$ and $\tilde\gamma_2(v_n)\to 0$. Passing to a subsequence, we may assume $v_n\rightharpoonup v$ in $\s$ and $v_n\to v$ in $L^2(\Omega)$. By lower weak semi-continuity we have
\[\tilde\gamma_2(v)\le\liminf_n\tilde\gamma_2(v_n)=0,\]
in particular $\|v\|^2\le\lambda_{h+1}\|v\|_2^2$. Recalling the definition of $\lambda_{h+1}$, in fact we have $\|v\|^2=\lambda_{h+1}\|v\|_2^2$. Two cases may occur:
\begin{itemize}[leftmargin=1cm]
\item[$(a)$] if $v=0$, then we have
\[\|v_n\|^2=\tilde\gamma_2(v_n)+\int_\Omega\eta_2(x)v_n^2\,dx\to 0,\]
against $\|v_n\|=1$;
\item[$(b)$] if $v\neq 0$, then $v$ is a $\lambda_{h+1}$-eigenfunction, in particular it has u.c.p.\ (Proposition \ref{ucp}), which in turn implies (recall that $\eta_2\not\equiv\lambda_{h+1}$)
\[\|v\|^2=\int_\Omega\eta_2(x)v^2\,dx<\lambda_{h+1}\|v\|_2^2,\]
a contradiction.
\end{itemize}
Now we use \eqref{ex-z5} and \eqref{ex-z6} into \eqref{ex-z4} and we get
\beq\label{ex-z7}
\langle\tilde\varphi(u),\hat u-\check u\rangle \ge c_2\|\hat u\|^2+c_1\|\check u\|^2 \ge c_3\|u\|^2,
\eeq
where we have set $c_3=\min\{c_1,c_2\}>0$ and used orthogonality.
\vskip2pt
\noindent
We prove a similar estimate for $\tilde\psi_\lambda$:
\beq\label{ex-z8}
\langle\tilde\psi_\lambda(u),\hat u-\check u\rangle\ge c_4\|u\|^2 \ (c_4>0).
\eeq
Indeed, by orthogonality and $\lambda_h<\lambda<\lambda_{h+1}$ we have
\begin{align*}
&\langle\tilde\psi_\lambda(u),\hat u-\check u\rangle = \|\hat u\|^2-\|\check u\|^2-\lambda(\|\hat u\|_2^2-\|\check u\|_2^2) \\
&\ge \Big(1-\frac{\lambda}{\lambda_{h+1}}\Big)\|\hat u\|^2+\Big(\frac{\lambda}{\lambda_h}-1\Big)\|\check u\|^2 \ge c_4\|u\|^2,
\end{align*}
where we have set
\[c_4=\min\Big\{1-\frac{\lambda}{\lambda_{h+1}},\,\frac{\lambda}{\lambda_h}-1\Big\} > 0.\]
By \eqref{ex-z7} and \eqref{ex-z8} we have for all $t\in[0,1]$
\begin{align*}
&\langle\tilde h'_u(t,u),\hat u-\check u\rangle = (1-t)\langle\tilde\varphi'(t,u),\hat u-\check u\rangle + t\langle\tilde\psi'_\lambda(t,u),\hat u-\check u\rangle \\
&\ge \big((1-t)c_3+tc_4\big)\|u\|^2 \ge c_5\|u\|^2,
\end{align*}
where we have set $c_5=\min\{c_3,c_4\}>0$. Now we can prove our claim, namely, that $0\in K(\tilde h(t,\cdot))$ is isolated uniformly with respect to $t\in[0,1]$. Arguing by contradiction, assume that there exist sequences $(t_n)$ in $[0,1]$, $(u_n)$ in $X\setminus\{0\}$, s.t.\ $h'_u(t_n,u_n)=0$ for all $n\ge 1$ and $u_n\to 0$ in $X$. In particular, $u_n\to 0$ in $C^0_\delta(\overline\Omega)$, so for all $n$ big enough we have $\|u_n\|_\infty\le\delta_0$ and the above inequality yields
\[\langle\tilde h'_u(t_n,u_n),\hat u_n-\check u_n\rangle \ge c_5\|u_n\|^2 >0,\]
a contradiction. Therefore, by homotopy invariance of critical groups (see \cite[Theorem 5.6]{C}), we have for all $q\ge 0$
\[C_q(\tilde\varphi,0)=C_q(\tilde\psi_\lambda,0).\]
Recalling \eqref{ex-z1} and the invariance of critical groups on dense subspaces, allows us to conclude.
\end{proof}

\noindent
We can now prove our existence result:

\begin{theorem}\label{ex}
If ${\bf H}_1$ \ref{h11} - \ref{h14} hold, then \eqref{dir} has at least one non-trivial solution $\bar u\in C^\alpha_\delta(\overline\Omega)$.
\end{theorem}
\begin{proof}
We are still assuming that $K(\varphi)$ is finite, otherwise the conclusion is trivial. From Lemmas \ref{ex-inf} and \ref{ex-z} we know that $C_k(\varphi,0)=0$ and $C_k(\varphi,\infty)\neq 0$. By Proposition \ref{mr} \ref{mr2}, there exists $\bar u\in K(\varphi)$ s.t.\ $C_k(\varphi,\bar u)\neq 0$. So, in particular $\bar u\neq 0$. By the definition of (weak) solution and Proposition \ref{reg} \ref{reg2} we have $\bar u\in C^\alpha_\delta(\overline\Omega)$ for a convenient $\alpha>0$ and $\bar u$ solves \eqref{dir}.
\end{proof}

\section{Multiplicity result}\label{sec4}

\noindent
In order to obtain multiple solutions for problem \eqref{dir}, we need to slightly modify our hypotheses on $f$, avoiding the case $k=1$ in ${\bf H}_1$ \ref{h13} and letting '$h=0$' in ${\bf H}_1$ \ref{h14}:

\begin{itemize}[leftmargin=1cm]
\item[${\bf H}_2$] $f:\Omega\times\R\to\R$ is s.t.\ $f(\cdot,t)$ is measurable in $\Omega$ for all $t\in\R$, $f(x,\cdot)\in C^1(\R)$ for a.e.\ $x\in\Omega$, we set for all $(x,t)\in\Omega\times\R$
\[F(x,t)=\int_0^t f(x,\tau)\,d\tau.\]
Morerover:
\begin{enumroman}
\item\label{h21} for all $\rho>0$ there exists $a_\rho\in L^\infty(\Omega)_+$ s.t.\ $|f(x,t)|\le a_\rho(x)$ for a.e.\ $x\in\Omega$ and all $|t|\le\rho$;
\item\label{h22} $\displaystyle\lim_{|t|\to\infty}\big(f(x,t)t-2F(x,t)\big)=+\infty$ uniformly for a.e.\ $x\in\Omega$;
\item\label{h23} there exists an integer $k\ge 2$ s.t.\ uniformly for a.e.\ $x\in\Omega$
\[\lambda_k\le\liminf_{|t|\to \infty}\frac{f(x,t)}{t}\le\limsup_{|t|\to\infty}\frac{f(x,t)}{t}\le\lambda_{k+1};\]
\item\label{h24} there exists a function $\eta_0\in L^\infty(\Omega)$ s.t.\ $0\le\eta_0(x)\le\lambda_1$ for a.e.\ $x\in\Omega$, with $\eta_0\not\equiv\lambda_1$, and $\delta_0>0$ s.t.\ for a.e.\ $x\in\Omega$ and all $|t|\le\delta_0$
\[\frac{F(x,t)}{t^2}\le\frac{\eta_0(x)}{2}.\]
\end{enumroman}
\end{itemize}

\begin{example}
Let $f\in C^1(\R)$ be defined by
\[f(t)=\begin{cases}
\lambda_k t-(\mu-\lambda_k)\Big(\frac{1}{2}\ln|t|+\sqrt{|t|}\big) & \text{if $t<-1$} \\
\mu t & \text{if $|t|\le 1$} \\
\lambda_k t+(\mu-\lambda_k)\Big(\frac{1}{2}\ln|t|+\sqrt{|t|}\big) & \text{if $t>1$,}
\end{cases}\]
where $\mu\in(0,\lambda_1)$, and $k\ge 2$ is an integer. Then $f$ satisfies ${\bf H}_2$ \ref{h21} - \ref{h24}.
\end{example}

\noindent
We introduce truncated energy functionals by setting for all $(x,t)\in\Omega\times\R$
\[f_\pm(x,t)=f(x,\pm t^\pm), \ F_\pm(x,t)=\int_0^t f_\pm(x,\tau)\,d\tau,\]
and for all $u\in\s$
\[\varphi_\pm(u)=\frac{\|u\|^2}{2}-\int_\Omega F_\pm(x,u)\,dx.\]
We establish now some properties of these functionals:

\begin{lemma}\label{mu-c}
If ${\bf H}_2$ \ref{h21} - \ref{h24} hold, then $\varphi_\pm\in C^1(\s)$ and satisfy \eqref{c}. Moreover, if $u\in K(\varphi_\pm)\setminus\{0\}$, then $u\in\pm{\rm int}\,(C^0_\delta(\overline\Omega)_+)$ and $u$ is a solution of \eqref{dir}.
\end{lemma}
\begin{proof}
We only deal with $\varphi_+$ (the argument for $\varphi_-$ is similar). By ${\bf H}_2$ \ref{h24} we have $f(x,0)=0$ for a.e.\ $x\in\Omega$, hence $f:\Omega\times\R\to\R$ is a Carath\'eodory mapping. So, $\varphi_+$ turns out to be of class $C^1$ in $\s$ (note that the truncation produces a loss of regularity) with derivative given for all $u,v\in\s$ by
\[\varphi'_+(u)(v)=\langle u,v\rangle -\int_\Omega f_+(x,u)v\,dx.\]
We prove that $\varphi_+$ sasitsfies \eqref{c}: let $(u_n)$ be a sequence in $\s$ s.t.\ $|\varphi_+(u_n)|\le M$ for all $n\ge 1$, and $(1+\|u_n\|)\varphi'_+(u_n)\to 0$ in $H^{-s}(\Omega)$. Then there exists $(\eps_n)$ in $\R$ s.t.\ $\eps_n\to 0^+$ and for all $n\ge 1$, $v\in\s$
\beq\label{mu-c1}
\Big|\langle u_n,v\rangle -\int_\Omega f_+(x,u_n)v\,dx\Big|\le\frac{\eps_n\|v\|}{1+\|u_n\|}.
\eeq
Choosing $v=-u_n^-$ in \eqref{mu-c1} yields
\[\|u_n^-\| \le \langle u_n,-u_n^-\rangle \le \frac{\eps_n\|u_n^-\|}{1+\|u_n\|},\]
and the latter tends to $0$ as $n\to\infty$, so $u^-_n\to 0$ in $\s$. We claim now that $(u^+_n)$ is bounded in $\s$. We argue by contradiction, assuming that $\|u^+_n\|\to\infty$. We set $w_n=u^+_n\|u^+_n\|^{-1}$, so $\|w_n\|=1$, and $w_n(x)\ge 0$ for all $n\ge 1$ and a.e.\ $x\in\Omega$. Passing to a subsequence, we have $w_n\rightharpoonup w$ in $\s$ and $w_n\to w$ in $L^2(\Omega)$, in particular $w(x)\ge 0$ for a.e.\ $x\in\Omega$. By \eqref{mu-c1} and $\|u^-_n\|\to 0$, we can find a sequence $(\eps'_n)$ s.t.\ $\eps'_n\to 0$ and for all $n\ge 1$, $v\in\s$
\[\Big|\langle u^+_n,v\rangle -\int_\Omega f_+(x,u^+_n)v\,dx\Big|\le\eps'_n\|v\|.\]
Dividing by $\|u^+_n\|$ we get
\beq\label{mu-c2}
\Big|\langle w_n,v\rangle -\int_\Omega \frac{f_+(x,u^+_n)}{\|u^+_n\|}v\,dx\Big|\le\frac{\eps'_n\|v\|}{\|u^+_n\|}.
\eeq
Reasoning as in the proof of Lemma \ref{ex-c}, from ${\bf H}_2$ \ref{h23} we deduce the existence of $\eta_+\in L^\infty(\Omega)$ s.t.\ $\lambda_k\le\eta_+(x)\le\lambda_{k+1}$ for a.e.\ $x\in\Omega$ and
\[\frac{f_+(\cdot, u^+_n)}{\|u^+_n\|}\rightharpoonup\eta_+ w \ \text{in $L^2(\Omega)$.}\]
Choose $v=w_n-w$ in \eqref{mu-c2}, then
\[\big|\langle w_n,w_n-w\rangle\big| \le \Big|\int_\Omega\frac{f_+(x,u^+_n)}{\|u^+_n\|}(w_n-w)\,dx\Big|+\frac{\eps'_n\|w_n-w\|}{\|u^+_n\|},\]
and the latter tends to $0$ as $n\to\infty$. As before, we get $w_n\to w$ in $\s$, in particular $\|w\|=1$. Moreover, passing to the limit in \eqref{mu-c2} we see that $w$ is a solution of the \eqref{ev}-type problem
\beq\label{mu-c3}
\begin{cases}
\fl w=\eta_+(x)w & \text{in $\Omega$} \\
w=0 & \text{in $\Omega^c$.}
\end{cases}
\eeq
So we have $1\in\sigma(\eta_+)$ with $w$ as an associated eigenfunction. Since $\eta_+(x)\ge\lambda_k$ for a.e.\ $x\in\Omega$, we have
\[\lambda_k(\eta_+)\le\lambda_k(\lambda_k)=1,\]
and (recalling that $k\ge 2$) we deduce that $w$ is a non-principal $\eta_+$-eigenfunction. By Proposition \ref{spe} \ref{spe2} we know that $w$ must change sign, a contradiction. Thus, $(u^+_n)$ is bounded in $\s$.
\vskip2pt
\noindent
So far we have proved that $u_n=u^+_n-u^-_n$ form a bounded sequence in $\s$. Passing to a subsequence, $u_n\rightharpoonup u$ in $\s$ and $u_n\to u$ in $L^2(\Omega)$. Choosing $v=u_n-u$ in \eqref{mu-c1} then leads to $u_n\to u$ in $\s$.
\vskip2pt
\noindent
Finally, let $u\in K(\varphi_+)\setminus\{0\}$. We have for all $v\in\s$
\beq\label{mu-c4}
\langle u,v\rangle=\int_\Omega f_+(x,u)v\,dx.
\eeq
Choosing $v=-u^-\in\s$ yields
\[\|u^-\|^2\le\langle u,-u^-\rangle = \int_\Omega f_+(x,u)(-u^-)\,dx = 0,\]
hence $u^-=0$, i.e., $u(x)\ge 0$ for a.e.\ $x\in\Omega$. By ${\bf H}_2$ \ref{h23}, \ref{h24} we see that there is $c>0$ s.t.\ $f_+(x,t)\ge -ct$ for a.e.\ $x\in\Omega$ and all $t\in\R^+$, so by Proposition \ref{hl} we have $u\in{\rm int}\,(C^0_\delta(\overline\Omega)_+)$. In particular, $u(x)>0$ for all $x\in\Omega$, so in \eqref{mu-c4} we can replace $f_+(x,u)$ with $f(x,u)$ and $u$ turns out to be a solution of \eqref{dir}.
\end{proof}

\noindent
The following lemmas determine the structure of the critical sets of functionals $\varphi_\pm$. Without loss of generality, we shall assume henceforth that $K(\varphi_\pm)$ is {\em finite}.

\begin{lemma}\label{mu-min}
If ${\bf H}_2$ \ref{h21} - \ref{h24} hold, then $\varphi_\pm$ has a strict local minimum at $0$.
\end{lemma}
\begin{proof}
We only deal with $\varphi_+$ (the argument for $\varphi_-$ is similar). Fix $\eps>0$, to be defined more precisely later, and $r\in(2,2^*_s)$. By ${\bf H}_2$ \ref{h22} - \ref{h24} there exists $C_\eps>0$ s.t.\ for a.e.\ $x\in\Omega$ and all $t\in\R$
\beq\label{mu-min1}
F_+(x,t)\le\frac{\eta_0(x)+\eps}{2} t^2+C_\eps |t|^r.
\eeq
We claim that there exists $c_0>0$ s.t.\ for all $u\in\s$
\beq\label{mu-min2}
\|u\|^2-\int_\Omega\eta_0(x)u^2\,dx\ge c_0\|u\|^2.
\eeq
Indeed, by $\eta_0(x)\le\lambda_1$ for a.e.\ $x\in\Omega$ and Proposition \ref{spe} \ref{spe1} we have
\[\|u\|^2-\int_\Omega\eta_0(x)u^2\,dx\ge 0.\]
Now we argue by contradiction, assuming that there is a sequence $(u_n)$ in $\s$ s.t.\ $\|u_n\|=1$ and
\[\|u_n\|^2-\int_\Omega\eta_0(x)u_n^2\,dx\to 0.\]
Passing to a subsequence, we have $u_n\rightharpoonup u$ in $\s$ and $u_n\to u$ in $L^2(\Omega)$. By lower semi-continuity we have
\[\|u\|^2 \le \liminf_n\|u_n\|^2 = \int_\Omega\eta_0(x)u^2\,dx \le \lambda_1\|u\|_2^2.\]
Two cases may occur:
\begin{itemize}[leftmargin=1cm]
\item[$(a)$] if $u=0$, then $u_n\to 0$ in $\s$, against $\|u_n\|=1$;
\item[$(b)$] if $u\neq 0$, then $u$ is a $\lambda_1$-eigenfunction, hence it has u.c.p.\ (Proposition \ref{ucp}) and
\[\|u\|^2= \int_\Omega\eta_0(x)u^2\,dx < \lambda_1\|u\|_2^2,\]
a contradiction.
\end{itemize}
Using \eqref{mu-min1} and \eqref{mu-min2}, and recalling the continuous embeddings $\s\hookrightarrow L^2(\Omega),\,L^r(\Omega)$, we get for all $u\in\s$
\begin{align*}
\varphi_+(u) &\ge \frac{\|u\|^2}{2}-\int_\Omega\Big(\frac{\eta_0(x)+\eps}{2}u^2+C_\eps|u|^r\Big)\,dx \\
&\ge \frac{1}{2}\Big(\|u\|^2-\int_\Omega\eta_0(x)u^2\,dx\Big)-\eps c_1\|u\|^2-c_2\|u\|^r \\
&\ge \Big(\frac{c_0}{2}-\eps c_1\Big)\|u\|^2-c_2\|u\|^r.
\end{align*}
Choosing $\eps\in(0,c_0(2c_1)^{-1})$, we see that the mapping
\[t\mapsto\Big(\frac{c_0}{2}-\eps c_1\Big)t^2-c_2t^r\]
is positive in $(0,\rho_0)$ for $\rho_0>0$ small enough. Then, for all $u\in\s$ with $\|u\|<\rho_0$ we have $\varphi_+(u)>0$. Thus, $0$ is a strict local minimizer for $\varphi_+$.
\end{proof}

\begin{lemma}\label{mu-mp}
If ${\bf H}_2$ \ref{h21} - \ref{h24} hold, then $\varphi_\pm$ has a critical point $u_\pm\in\pm{\rm int}\,(C^0_\delta(\overline\Omega)_+)$ of mountain pass type.
\end{lemma}
\begin{proof}
As usual we only deal with $\varphi_+$. From Lemma \ref{mu-min} and \eqref{c} we deduce the existence of $r>0$ s.t.
\[\inf_{\|u\|=r}\varphi_+(u)=m_r>0.\]
Besides, by ${\bf H}_2$ \ref{h23} for all $\beta\in(\lambda_1,\lambda_k)$ there exists $\rho>0$ s.t.\ for a.e.\ $x\in\Omega$ and all $t>\rho$
\[F_+(x,t)\ge\frac{\beta t^2}{2}.\]
Let $e_1\in\s$ be defined as in Proposition \ref{spe}: using ${\bf H}_2$ \ref{h21} we have for all $\tau>0$ and some $c_3>0$
\begin{align*}
\varphi_+(\tau e_1) &= \frac{\tau^2\|e_1\|^2}{2}-\int_\Omega F_+(x,\tau e_1)\,dx \\
&\le \frac{\tau^2\|e_1\|^2}{2}-\int_{\{\tau e_1>\rho\}}\frac{\beta}{2}(\tau e_1)^2\,dx+\int_{\{\tau e_1\le\rho\}} a_\rho(x)\rho\,dx \\
&\le \frac{(\lambda_1-\beta)\tau^2}{2}+c_3,
\end{align*}
and the latter tends to $-\infty$ as $\tau\to\infty$. Thus, for $\tau>0$ big enough we have $\varphi_+(\tau e_1)<0$. We apply then a variant of the Mountain Pass Theorem (see \cite{H}) to conclude that there exists $u_+\in K(\varphi_+)$ of mountain pass type, s.t.\ $\varphi_+(u_+)\ge m_r$, in particular $u_+\neq 0$. By Lemma \ref{mu-c} we conclude.
\end{proof}

\noindent
Now we go back to $\varphi$:

\begin{lemma}\label{mu-g}
If ${\bf H}_2$ \ref{h21} - \ref{h24} hold, then $0,u_+,u_-\in K(\varphi)$, and for all $q\ge 0$
\begin{enumroman}
\item\label{mu-ga} $C_q(\varphi,0)=\delta_{q,0}\R$;
\item\label{mu-gb} $C_q(\varphi,u_+)=\delta_{q,1}\R$;
\item\label{mu-gc} $C_q(\varphi,u_-)=\delta_{q,1}\R$.
\end{enumroman}
\end{lemma}
\begin{proof}
Reasoning as in Lemma \ref{mu-min}, it can be seen that $0$ is as well a strict local minimizer of $\varphi$, so by Proposition \ref{cg-min} we have \ref{mu-ga}.
\vskip2pt
\noindent
From Lemma \ref{mu-mp} and Proposition \ref{cg-mp1} we know that $C_1(\varphi_+,u_+)\neq 0$. Set for all $(t,u)\in[0,1]\times\s$
\[h_+(t,u)=(1-t)\varphi(u)+t\varphi_+(u).\]
Then, for all $t\in[0,1]$ $h_+(t,\cdot)\in C^1(\s)$ and $u_+\in K(h_+(t,\cdot))$. We claim that $u_+$ is an isolated critical point of $h_+(t,\cdot)$, uniformly with respect to $t\in[0,1]$. Arguing by contradiction, assume that there exist sequences $(t_n)$ in $[0,1]$ and $(u_n)$ in $\s\setminus\{u_+\}$, respectively, s.t.\ $u_n\in K(h_+(t_n,\cdot))$ for all integer $n\ge 1$ and $u_n\to u_+$ in $\s$. So, for all $n\ge 1$ $u_n$ solves the problem
\beq\label{mu-g1}
\begin{cases}
\fl u_n=(1-t_n)f(x,u_n)+t_nf_+(x,u_n) & \text{in $\Omega$} \\
u_n=0 & \text{in $\Omega^c$.}
\end{cases}
\eeq
By ${\bf H}_2$ \ref{h21} and Proposition \ref{reg}, the sequence $(u_n)$ is bounded in $C^\alpha_\delta(\overline\Omega)$. So, by the compact embedding $C^\alpha_\delta(\overline\Omega)\hookrightarrow C^0_\delta(\overline\Omega)$, passing to a subsequence we have $u_n\to u_+$ in $C^0_\delta(\overline\Omega)$. By Lemma \ref{mu-mp}, for $n\ge 1$ big enough we have $u_n\in{\rm int}\,(C^0_\delta(\overline\Omega)_+)$, in particular $u_n(x)>0$ for all $x\in\Omega$. So, \eqref{mu-g1} reduces to
\[\begin{cases}
\fl u_n=f_+(x,u_n) & \text{in $\Omega$} \\
u_n=0 & \text{in $\Omega^c$,}
\end{cases}\]
i.e., $u_n\in K(\varphi_+)$. Thus, $\varphi_+$ has infinitely many critical points, a contradiction.
\vskip2pt
\noindent
Homotopy invariance of critical groups (see \cite[Theorem 5.6]{C}) then yields $C_1(\varphi,u_+)\neq 0$. Now recall that (different from $\varphi_+$) the functional $\varphi$ is of class $C^2$ with finite Morse index and nullity at $u_+$, so by Proposition \ref{cg-mp2} the latter condition improves to \ref{mu-gb}.
\vskip2pt
\noindent
The argument for \ref{mu-gc} is analogous to that for \ref{mu-gb} (with $\varphi_+$ replaced by $\varphi_-$), so the proof is concluded.
\end{proof}

\noindent
Finally, we can prove our multiplicity result:

\begin{theorem}\label{mu}
If ${\bf H}_2$ \ref{h21} - \ref{h24} hold, then \eqref{dir} has at least three non-trivial solutions $u_+,u_-,\tilde u\in C^\alpha_\delta(\overline\Omega)$, with $u_-(x)<0<u_+(x)$ for all $x\in\Omega$.
\end{theorem}
\begin{proof}
Clearly we may assume that $K(\varphi)$ is finite. From Lemma \ref{mu-g} we already know that $0,u_+,u_-\in K(\varphi)$ and that, in particular,
\beq\label{mu1}
C_k(\varphi,0)=C_k(\varphi,u_+)=C_k(\varphi,u_-)=0
\eeq
(recall that $k\ge 2$ in ${\bf H}_2$ \ref{h23}). Besides, reasoning as in Lemma \ref{ex-inf} (whose hypotheses are all included in ${\bf H}_2$), we see that $C_k(\varphi,\infty)\neq 0$. Now Proposition \ref{mr} \ref{mr2} ensures the existence of $\tilde u\in K(\varphi)$ s.t.\ $C_k(\varphi,\tilde u)\neq 0$, which by \eqref{mu1} implies $\tilde u\notin\{0,u_+,u_-\}$.
\vskip2pt
\noindent
Finally, Lemma \ref{mu-mp} and Proposition \ref{reg} yield $u_+,u_-,\tilde u\in C^\alpha_\delta(\overline\Omega)$, $u_-(x)<0<u_+(x)$ for all $x\in\Omega$, and all three are solutions of \eqref{dir}.
\end{proof}

\noindent
{\small {\bf Aknowledgement.} A.\ Iannizzotto is a member of the Gruppo Nazionale per l'Analisi Matematica, la Probabilit\`a e le loro Applicazioni (GNAMPA) of the Istituto Nazionale di Alta Matematica (INdAM) and is partially supported by the GNAMPA research project {\em Regolarit\`a, esistenza e propriet\`a geometriche per le soluzioni di equazioni con operatori frazionari non lineari}. This work was partially performed during a visit of N.S.\ Papageorgiou at the University of Cagliari, funded by the Visiting Professor Programme.}

\end{document}